\documentclass[12pt]{article}

\usepackage{amssymb,amsmath,amscd}
\usepackage{mathrsfs}
\usepackage{amsfonts}
\usepackage{amssymb,amsthm}
\usepackage[arc,all]{xy}

\title{Correspondence between 2 Calabi-Yau Categories and Quivers}

\author{Jie Ren}

\begin{document}
\maketitle

\newtheorem*{thm}{Theorem}
\newtheorem*{rem}{Remark}

\def\dim{{\rm dim}}

In Section 8 of \cite{KoSo1}, M. Kontsevich and Y. Soibelman proved
that the equivalence classes of a certain type of 3-dimensional
Calabi-Yau categories are in one-to-one correspondence with the
gauge equivalence classes of quiver with minimal potential $(Q, W)$.
This note gives an analogue in 2-dimensional Calabi-Yau case. We
assume that \textbf{k} is a field of characteristic zero.

\begin{thm}
Let $\mathscr{C}$ be a 2-dimensional
\textbf{k}-linear Calabi-Yau category generated by a finite
collection $\mathcal {E}=\{E_{i}\}_{i\in I}$ of generators
satisfying
\begin{itemize}
\item[$\bullet$]$Ext^{0}(E_{i},E_{i})=\textbf{k} \cdot id_{E_{i}}$,
\item[$\bullet$]$Ext^{0}(E_{i},E_{j})=0, \forall i\neq j$,
\item[$\bullet$]$Ext^{<0}(E_{i},E_{j})=0, \forall i, j$.
\end{itemize}
The equivalence classes of such categories with respect to
$A_{\infty}$-transformations preserving the Calabi-Yau structure and
$\mathcal {E}$, are in one-to-one correspondence with finite
symmetric quivers with even number of loops at each vertex.
\end{thm}

\begin{proof}  Let's denoted by $\mathscr{A}$ the set of equivalence classes of such 2 Calabi-Yau categories, and
$\mathscr{B}$ the set of finite symmetric quivers with even number
of loops at each vertex.

Given such a category $\mathscr{C}$, we associate a quiver $Q$ whose
vertices $\{i\}_{i\in I}$ are in one-to-one correspondence with
$\mathcal {E}=\{E_{i}\}_{i\in I}$, and the number of arrows from $i$
to $j$ is equal to $\dim Ext^{1}(E_{i}, E_{j})$. Since $\mathscr{C}$
is 2 Calabi-Yau, we have $\dim Ext^{1}(E_{i}, E_{j})=\dim
Ext^{1}(E_{i}, E_{j})^{\vee}=\dim Ext^{1}(E_{j}, E_{i})$, so $Q$ is
symmetric. The supersymmetric non-degenerate pairing on
$Ext^{\bullet}(E_{i},E_{i})$ leads to a symplectic pairing on
$Ext^{1}(E_{i},E_{i})$, thus $\dim Ext^{1}(E_{i},E_{i})$ is even,
which means that the number of loops at each vertex is even. This
construction defines a map $\Phi:
\mathscr{A}\rightarrow\mathscr{B}$.

To prove that $\Phi$ is a bijection, we consider $\mathscr{C}$ with
single generator $E$, and quiver $Q$ with single vertex for
simplicity. The general case can be proved in a similar way.

Let $Q$ be a quiver with one vertex and $|J|=2n$ loops, where $J$ is
the set of loops. We will construct a 2 Calabi-Yau category with one
generator $E$, such that $2n=$dim $Ext^{1}(E,E)$. Assuming that such
a category exists, we will find an explicit formula for the
potential on $A=Hom^{\bullet}(E,E)$. Let's consider the graded
vector space
$$Ext^{\bullet}(E,E)[1]=Ext^{0}(E,E)[1]\oplus Ext^{1}(E,E)\oplus Ext^{2}(E,E)[-1]=\textbf{k}[1]\oplus \textbf{k}^{2n}\oplus
\textbf{k}[-1].$$

We introduce graded coordinates on $Ext^{\bullet}(E,E)[1]$:
\begin{itemize}
\item[a)]the coordinate $\alpha$ of degree 1 on $Ext^{0}(E,E)[1]$,
\item[b)]the coordinate $\beta$ of degree $-1$ on $Ext^{2}(E,E)[-1]$,
\item[c)]the coordinates $x_{i}, \xi_{i}, i=1,...,n$ of degree 0 on $Ext^{1}(E,E)=Ext^{1}(E,E)^{\vee}$.
\end{itemize}
The Calabi-Yau structure gives rise to the minimal potential
$W=W(\alpha, x_{i}, \xi_{i}, \beta)$, which is a series of cyclic
words on the space $Ext^{\bullet}(E,E)[1]$. Furthermore, $A$ defines
a non-commutative formal  pointed graded manifold endowed with a
symplectic structure (c.f. \cite{KoSo2}). The potential $W$
satisfies the equation $\{W,W\}=0$, where $\{\bullet,\bullet\}$ is
the corresponding Poisson bracket.

We need to construct the formal series $W$ of degree 1 in cyclic
words on the graded vector space $\textbf{k}[1]\oplus
\textbf{k}^{2n}\oplus \textbf{k}[-1]$, satisfying $\{W,W\}=0$ with
respect to the Poisson bracket
$$\{f,g\}=\sum_{i=1}^{n}[\frac{\partial}{\partial x_{i}},\frac{\partial}{\partial
\xi_{i}}](f,g)+[\frac{\partial}{\partial\alpha},\frac{\partial}{\partial\beta}](f,g).$$

Let $W_{can}=\alpha^{2}\beta+\sum_{i=1}^{n}(\alpha
x_{i}\xi_{i}-\alpha\xi_{i}x_{i})$. This potential makes
$Ext^{\bullet}(E,E)$ into a 2 Calabi-Yau algebra with associative
product and the unit. The multiplications are as follows: the
multiplication of $Ext^{0}(E,E)$ and the other components is scalar
product, and is a non-degenerate bilinear form on the components
$Ext^{1}(E,E)\otimes Ext^{1}(E,E)\rightarrow Ext^{2}(E,E)\simeq
\textbf{k}$.

In addition,
\begin{equation}
\begin{array}{ll}
\{W_{can},W_{can}\}&=\sum\limits_{i=1}^{n}[\frac{\partial
W_{can}}{\partial x_{i}},\frac{\partial W_{can}}{\partial
\xi_{i}}]+[\frac{\partial W_{can}}{\partial\alpha},\frac{\partial
W_{can}}{\partial\beta}]
\\&=\sum\limits_{i=1}^{n}(\xi_{i}\alpha-\alpha\xi_{i})(\alpha
x_{i}-x_{i}\alpha)-(\alpha
x_{i}-x_{i}\alpha)(\xi_{i}\alpha-\alpha\xi_{i})
\\&+(\alpha\beta+\beta\alpha+\sum\limits_{j=1}^{n}(x_{j}\xi_{j}-\xi_{j}x_{j}))\alpha^{2}-\alpha^{2}(\alpha\beta+\beta\alpha+\sum\limits_{k=1}^{n}(x_{k}\xi_{k}-\xi_{k}x_{k}))
\\&=0\nonumber
\end{array}
\end{equation}

The above construction from $Q$ to $\mathscr{C}$ shows that $\Phi$
is a surjection.

Finally, we need to check that $\Phi$ is an injection. The 2
Calabi-Yau algebras we are considering can be thought of as
deformations of the 2 Calabi-Yau algebra
$A_{can}=Ext^{\bullet}(E,E)$ corresponding to the potential
$W_{can}$. The deformation theory of $A_{can}$ is controlled by a
differential graded Lie algebra (DGLA)
$\mathfrak{g}_{can}=\bigoplus_{n\in\mathbb{Z}}\mathfrak{g}_{can}^{n}$,
which is a DG Lie subalgebra of the DGLA
$\widehat{\mathfrak{g}}=\prod_{k\geqslant1}Cycl^{k}(A_{can}[1])^{\vee}=\bigoplus_{n\in\mathbb{Z}}\widehat{\mathfrak{g}}^{n}$.
Here we write $\widehat{\mathfrak{g}}^{n}=\{W| coh.deg W=n\}$, and
$\mathfrak{g}_{can}^{n}=\{W\in\widehat{\mathfrak{g}}^{n}| cyc.deg
W\geqslant n+2\}$, where coh.deg means the cohomological degree of
$W$, and cyc.deg means the number of letters $\alpha, x_{i},
\xi_{i}, \beta, i=1,...,n$ that each term of $W$ contains. In these DGLAs, the
Lie bracket is given by the Poisson bracket and the differential is
given by $d=\{W_{can}, \bullet\}$. The DGLA $\mathfrak{g}_{can}$ is
a DG Lie subalgebra of $\widehat{\mathfrak{g}}$ bdcause of the following reason: $d$ preserves $\mathfrak{g}_{can}$ since it increases
both coh.deg and cyc.deg by 1, and the Poisson bracket restricts to $\mathfrak{g}_{can}$ since for any $W_1\in\mathfrak{g}_{can}^m$ and $W_2\in\mathfrak{g}_{can}^n$ with $cyc.deg W_1=l_1$ and $cyc.deg W_2=l_2$, we have $cyc.deg \{W_1,W_2\}=l_1+l_2-2\geqslant m+2+n+2-2=m+n+2$. As vector spaces,
$\widehat{\mathfrak{g}}=\mathfrak{g}_{can}\bigoplus\mathfrak{g}$,
where $\mathfrak{g}=\bigoplus_{n\in\mathbb{Z}}\mathfrak{g}^{n}$, and
$\mathfrak{g}^{n}=\{W\in\widehat{\mathfrak{g}}^{n}|cyc.deg W<n+2\}$.
For the same reason as $\mathfrak{g}_{can}$, we have that
$\mathfrak{g}$ is also a DG Lie subalgebra of
$\widehat{\mathfrak{g}}$. It follows that $\mathfrak{g}_{can}$ is a
direct summand of the complex $\widehat{\mathfrak{g}}$. The latter
is quasi isomorphic to the cyclic complex
$CC_{\bullet}(A_{can})^{\vee}$. Let $A_{can}^{+}\subset A_{can}$ be
the non-unital $A_{\infty}$-subalgebra consisting of terms of
positive cohomological degree. Then for the cyclic homology,
$HC_{\bullet}(A_{can})\simeq HC_{\bullet}(A_{can}^{+})\bigoplus
HC_{\bullet}(\textbf{k})$. Thus to compute the cohomology of the dual complex
$\widehat{\mathfrak{g}}$, we only need to consider the space of cyclic series in variables $x_{i},
\xi_{i}, \beta, i=1,...,n$ (corresponds to
$HC_{\bullet}(A_{can}^{+})^{\vee}$), and the one in variable
$\alpha$ (corresponds to $HC_{\bullet}(\textbf{k})^{\vee}$). We have
that the series in $\alpha$ don't contribute to the cohomology of
$\mathfrak{g}_{can}$ since they do not belong to $\mathfrak{g}_{can}$.
Moreover, the cohomological degree of series in $x_{i}, \xi_{i},
\beta, i=1,...,n$ is non-positive. Hence
$H^{\geqslant1}(\mathfrak{g}_{can})=0$, which means that the deformation of
$A_{can}$ is trivial. Thus, $\Phi$ is an injection.
\end{proof}

\begin{rem}
Suppose that the 2 Calabi-Yao category over an algebraically closed field is endowed with a stability condition, then a polystable object has formal endomorphism algebra. A special case of the coherent sheaves on a projective K3 surface is proven in \cite{BZ}.
\end{rem}

Address:\\Department of Mathematics, KSU, Manhattan, KS 66506, USA,
\\jren@ksu.edu

\end{document}